\documentclass[11pt]{article}

\usepackage{styles}
\bluehyperref
\usepackage[numbers]{natbib}
\usemedgeometry

\usepackage{overpic}
\usepackage{enumitem}

\usepackage{macros}

\newcommand{\dfunc}{\mathsf{d}}  
\newcommand{\id}{\textup{Id}}
\newcommand{\pf}{\mathop{\textup{pf}}}
\newcommand{\sgn}{\mathop{\textup{sgn}}}

\begin{document}

\title{Random tilts to find stationary points in \\
  stochastic convex optimization}
\author{Felipe Areces, John C.\ Duchi, Malo Sommers}

\maketitle

\begin{abstract}
  We consider the problem of finding stationary points of
  stochastic convex functions and related variational inequalities.
  For each, we show that regularized empirical risk minimization, coupled
  with a random tilting perturbation, obtains stationarity residual order
  $\sqrt{d/n}$ for $d$-dimensional problems given $n$ observations.
  We present a few complementary results that show that some dimension
  dependence is necessary, in distinction from standard stochastic
  optimization and empirical risk minimization, by providing minimax lower
  bounds scaling as $\sqrt{\log d / n}$.
\end{abstract}

\section{Introduction}

In the paper~\cite{ArecesDuSo26}, \citeauthor{ArecesDuSo26} study the
problem of finding stationary points of potentially non-smooth
stochastic convex functions.
Taking as motivation problems in distribution-free
inference~\cite{VovkGaSh05,LeiWa14,AngelopoulosBaBa24}, as well as
more general variational inequalities, they consider the (stochastic)
convex optimization problem of solving
\begin{align}
  \label{eqn:optimization-problem}
  \begin{array}
{rl} \minimize_x & f(x) \defeq \E_P[F_\statrv(x)]
    = \int F_\statval(x) dP(\statval) \\
    \subjectto & x \in X,
  \end{array}
\end{align}
where $X \subset \R^d$ is a closed convex set, and $F_\statval$ are
convex functions of random samples $\statrv \sim P$.

Letting
\begin{equation*}
  \normalcone_X(x) \defeq \left\{v \in \R^d \mid \<v, y - x\> \le 0
  ~ \mbox{for~all~} y \in X \right\}
\end{equation*}
denote the normal cone to $X$ at $x \in X$ (and
$\normalcone_X(x) = \emptyset$ otherwise), the optimality conditions
for the problem~\eqref{eqn:optimization-problem} are
precisely~\cite{HiriartUrrutyLe93} that
\begin{equation*}
  0 \in \partial f(x\opt) + \normalcone_X(x\opt).
\end{equation*}
Instead of the typical objective of optimality gap, the stationary
point problem takes this inclusion as the jumping off point to measure
error in terms of the \emph{stationarity gap}, which, for an estimator
$\what{x}_n$ given a sample $\statrv_1^n \simiid P$ of size $n$, is
\begin{align}
  \label{eqn:stationary-residual}
  \dist\left(0, \partial f(\what{x}_n) + \normalcone_X(\what{x}_n)
  \right).
\end{align}
Obtaining small stationarity gap is, however, somewhat non-trivial:
subdifferentials of convex functions fail to converge
uniformly~\cite{ArecesDuSo26,TianRo25}, and some natural algorithms,
such as regularized sample average minimization, can provably fail to
converge.

In spite of these difficulties, however, in this paper we show that a
minor modification by adding a small amount of randomization can
guarantee finding small stationary residual.
Describing the algorithm is simple: define the sampled function
$f_n(x) \defeq \frac{1}{n} \sum_{i = 1}^n F_{\statrv_i}(x)$, and let
$x_0 \in X$ be otherwise arbitrary.
Then for a vector $z \in \R^d$, define the tilted estimator
\begin{align}
  \label{eqn:estimator}
  x^\lambda_n(z) \defeq \argmin_{x \in X} \left\{f_n(x) - \<z, x\>
  + \frac{\lambda}{2} \ltwo{x - x_0}^2 \right\}.
\end{align}
By an appropriate choice of the randomized noise vector $Z$, we can
guarantee stationarity, and in this paper, we consider $Z \in \R^d$
with generalized $\laplace(\alpha)$ density
\begin{align}
  \label{eqn:generalized-Laplace}
  q_\alpha(z) \defeq c_\alpha e^{-\ltwo{z} / \alpha},
  ~~~
  c_\alpha^{-1} = \int_{\R^d} e^{-\ltwo{z} / \alpha} dz.
\end{align}

As a corollary to our main theorem to come, which generalizes the
setting to stochastic maximal monotone operators, we prove the
following result:
\begin{corollary}
  \label{corollary:general-convex}
  Assume that for each $\statval$, $F_\statval$ is $1$-Lipschitz with
  respect to the $\ell_2$-norm, and let $Z$ have generalized
  $\laplace(\alpha)$ density~\eqref{eqn:generalized-Laplace}.
  Then for any $x\opt \in \argmin_{x \in X} f(x)$,
  \begin{align*}
    \E\left[\dist\left(0, \partial f(x_n^\lambda(Z))
      + \normalcone_X(x_n^\lambda(Z))\right)\right]
    \le O(1) \left[\alpha d + \frac{1}{\alpha n}
      + \sqrt{\frac{d}{n}}\right]
    + 2 \lambda \ltwo{x_0 - x\opt}.
  \end{align*}
\end{corollary}
\noindent We prove a slightly more general result, considering maximal
monotone operators and variational inequalities, which form the
natural generalization of the subdifferential mapping of convex
functions.

In addition to these results, we also show that if we seek the
stationary residual~\eqref{eqn:stationary-residual}, then it is
impossible to avoid some dimension dependence in convergence
guarantees as in Corollary~\ref{corollary:general-convex}.
Taking $\lambda > 0$ but arbitrarily small and choosing
$\alpha = \frac{1}{\sqrt{dn}}$ evidently gives expected stationary
residual
\begin{align*}
  \E\left[\dist\left(0, \partial f(x_n^\lambda(Z))
    + \normalcone_X(x_n^\lambda(Z))\right)\right]
  \le O(1) \sqrt{\frac{d}{n}}.
\end{align*}
We do not currently know if this convergence rate is sharp, but we
show that some dimension dependence is necessary; informally (see
Theorem~\ref{theorem:lower-bound} to come), we show that there exist
$1$-Lipschitz convex losses such that
\begin{equation*}
  \E_P\left[\dist(0, \partial f(\what{x}_n) + \normalcone_X(\what{x}_n))
    \right] \gtrsim \sqrt{\frac{\log(d/n)}{n}},
\end{equation*}
for any estimator $\what{x}_n$.
Thus, some dimension dependence is necessary.

\subsection{The general convergence result}

We frame our results in a slightly more general setting than pure
convex optimization: we assume that we have a collection of maximal
monotone mappings $A_\statval : \R^d \toto \R^d$ for each
$\statval \in \statdomain$, and we let
\begin{align*}
  A(x) \defeq \E_P[A_\statrv(x)]
  = \int A_\statval(x) dP(\statval)
  ~~ \mbox{for}~~ \statrv \sim P,
\end{align*}
where we use the standard Aumann integral~\cite{AubinFr90} to define
the expectation; we shall not concern ourselves with measurability
issues.
Then we seek a solution to the variational inequality
\begin{equation}
  \label{eqn:vi}
  0 \in A(x) + \normalcone_X(x),
\end{equation}
where we assume that $X$ is closed convex and $A$ is maximal monotone,
and that the solution set $X\opt = (A + \normalcone_X)^{-1}(0)$ is
non-empty.
In this case, the immediate generalization of the
estimator~\eqref{eqn:estimator} chooses $x$ to solve
\begin{equation*}
  0 \in A_n(x)
  + \lambda (x - x_0)
  - z + \normalcone_X(x),
\end{equation*}
where $A_n(x) = \frac{1}{n} \sum_{i = 1}^n A_{\statrv_i}(x)$, that is,
\begin{align*}
  x_n^\lambda(z)
  = (A_n + \lambda (\id - x_0) + \normalcone_X)^{-1}(z).
\end{align*}
We will focus on $b$-bounded operators, meaning that for each
$\statval \in \statdomain$, $x \in \R^d$, and $a \in A_\statval(x)$,
we have $\ltwo{a} \le b$, as otherwise the integrability issues appear
to become highly non-trivial.
With these definitions, we have the following theorem.

\begin{theorem}
  \label{theorem:maximal-monotone}
  Let $A_\statval$ be $b$-bounded maximal monotone operators, and let
  $Z$ have generalized $\laplace(\alpha)$
  density~\eqref{eqn:generalized-Laplace}.
  Then
  \begin{equation*}
    \E\left[\dist(0, A(x_n^\lambda(Z)) + \normalcone_X(x_n^\lambda(Z)))\right]
    \le O(1) \left[\E[\ltwo{Z}^2]^{1/2}
      + \frac{b^2}{\alpha n} +
      b \sqrt{\frac{d}{n}}\right] + 2 \lambda \ltwo{x_0 - x\opt}.
  \end{equation*}
\end{theorem}
\noindent Corollary~\ref{corollary:general-convex} follows immediately
from the theorem: we simply take
$A_\statval(x) = \partial F_\statval(x)$ to be the subdifferential
mapping.
In this case, we require no assumption that $A$ is maximal monotone:
standard results of Bertsekas, Rockafellar, Wets, and others imply
that $\partial f(x) = \E_P[\partial F_\statrv(x)]$ and $\partial f$ is
maximal monotone~\cite{Bertsekas73,RockafellarWe98}.
Because $\E[\ltwo{Z}^2] = \alpha^2 d(d+1)$ for
$Z \sim \laplace(\alpha)$, we see that the choice
$\alpha = b / \sqrt{dn}$ gives
\begin{align*}
  \E\left[\dist(0, A(x_n^\lambda(Z)) + \normalcone_X(x_n^\lambda(Z)))
    \right]
  \lesssim b \sqrt{\frac{d}{n}} + \lambda \cdot \ltwo{x\opt - x_0}
\end{align*}
for any $\lambda > 0$.
While the result suggests that we ought to take
$\lambda \downarrow 0$, at least our arguments still require some
regularization to guarantee convergence.

\subsection{A few brief remarks}

\citet{ArecesDuSo26} provide a broader overview of stochastic
variational inequalities and finding stationary points, so we do not
recapitulate their background here.
In passing, however, we note the similarity between the tilted
estimator~\eqref{eqn:estimator} and objective perturbation in
differential privacy~\cite{ChaudhuriMoSa11, AgarwalKaSiTh23,
RedbergKoWa23}.
There, one adds random Gaussian or Laplacian noise to a regularized
empirical risk minimization problem, then uses this noise to argue
that the outputs $x_n^\lambda$ do not depend too strongly on any
individual observation $\statrv_i$.
The techniques in this line of work, at least for now, necessarily
fail here: because $F_\statval$ need not be differentiable, there are
positive probability events where $x_n^\lambda(z)$ identifies
precisely a particular observation via points of
non-differentiability.
It may be interesting to investigate if the techniques here could
point a way to obtaining stronger differential privacy guarantees for
objective perturbation.

Finally, while we do not have a complete understanding of the
dimension dependence in these problems, they do provide a sharp
departure from standard stochastic convex optimization problems, where
standard methods (such as stochastic gradient descent) achieve
``dimension-free'' rates that depend only on the $\ell_2$-radius of
the underlying set $X$ and the Lipschitz constant of the objectives
$F_\statval$.
In this case, however, we can extend the techniques of
\citet{ShalevShSrSr10} to prove the following complementary lower
bound:
\begin{theorem}
  \label{theorem:lower-bound}
  Assume that $X \subset \R^d$ has non-empty interior, $n\geq 3$, and
  $d/n\geq e^2$.
  Then there exists a collection of $1$-Lipschitz convex loss
  functions $\{F_\statval\}_{\statval \in \statdomain}$ such that
  taking population loss $f_P(x) = \E_P[F_\statrv(x)]$, for any
  potentially randomized procedure $\what{x}_n$,
  \begin{align*}
    \sup_P \E_P\left[\dist(0, \partial f_P(\what{x}_n)
      + \normalcone_X(\what{x}_n))\right]
    \ge c\min\left\{
      1,
      \sqrt{\frac{\log (d/n)}{n}}
    \right\}.
  \end{align*}
  where the supremum is over discrete probability distributions on
  $\statdomain$ and $c > 0$ is a numerical constant.
\end{theorem}
\noindent We provide the proof in Section~\ref{sec:proof-lower-bound}.

\paragraph{Notation}
Throughout this paper, $O(1)$ denotes a numerical (universal)
constant, whose value may change from line-to-line, but which is
independent of all problem parameters.
We also use $a_n \lesssim b_n$ to mean that there is a numerical
(universal) constant $C$ such that $a_n \le C b_n$, that is,
$a_n \le O(1) b_n$, for all $n \in \N$.

\section{Concentration of maximal monotone operators}
\label{sec:monotone-concentration}

To develop intuition, we note that we would like to simply
perform a change of variables and integrate: if we could ignore the
constraints $X$ and each $A_\statval$ were $\mc{C}^1$ and bounded, then the
mapping
\begin{equation*}
  T_{n,\lambda}(x) \defeq A_n(x) + \lambda (x - x_0)
\end{equation*}
is strongly monotonic and so is a continuously differentiable homeomorphism
from $\R^d \to \R^d$, with inverse $T_{n,\lambda}^{-1}(z)$.
Thus by a change of variables,
if $Z$ has density $q$ then
\begin{align*}
  \E\left[\ltwo{A_n(T_{n,\lambda}^{-1}(Z)) - A(T_{n,\lambda}^{-1}(Z))}^2
    \mid P_n \right]
  & = \int \ltwo{A_n(T_{n,\lambda}^{-1}(z)) - A(T_{n,\lambda}^{-1}(z))}^2
  q(z) dz \\
  & = \int_{\R^d} \ltwo{A_n(x) - A(x)}^2
  \det(\deriv T_{n,\lambda}(x)) q(T_{n,\lambda}(x)) dx,
\end{align*}
where we used $T_{n,\lambda}(x) = z$.
Now we recognize that $\det(B)$ is a homogeneous polynomial of degree at
most $d$ in the entries of the matrix $B \in \R^{d \times d}$, and so
because (at least heuristically) $\det(\deriv T_{n,\lambda}(x))$
should ``depend on'' at most $d$ observations $\statrv_i$ at a time,
we might hope that a concentration argument could give
\begin{align*}
  \E\left[\ltwo{A_n(x) - A(x)}^2
    \det(\deriv T_{n,\lambda}(x)) q(T_{n,\lambda}(x))\right]
  \le \frac{\textup{poly}(d)}{n}
  \E\left[\det(\deriv T_{n,\lambda}(x)) q(T_{n,\lambda}(x))
    \right].
\end{align*}
Were this true, then
making the change of variables
back via $z = T_{n,\lambda}(x)$, with volume
elements $dz = \det(\deriv T_{n,\lambda}(x)) dx$,
we would obtain
\begin{align*}
  \E\left[\ltwo{A_n(T_{n,\lambda}^{-1}(Z))
      - A(T_{n,\lambda}^{-1}(Z))}^2\right]
  & \le \frac{\textup{poly}(d)}{n}
  \E\bigg[\int \det(\deriv T_{n,\lambda}(x)) q(T_{n,\lambda}(x)) dx
    \bigg] \\
  & = \frac{\textup{poly}(d)}{n} \int q(z) dz
  = \frac{\textup{poly}(d)}{n}.
\end{align*}

In fact, we can make this approach go through by appealing
to some basic geometric measure theory---specifically,
the area formula---and expanding the
determinant as a sum of terms that only contain limited
``interaction'' with the sample.
This idea to expand a determinant in terms of its factors appears in the
literature on sampling determinantal-point-processes for empirical risk
minimization~\citet[cf.][]{DerezinskiWaHs22}, though our particular
expansion choices appear to be novel.
To follow the program above, we will require an approximation argument to
argue that we may work with continuously differentiable operators and
unconstrained minimizers; this is relatively standard convex and set-valued
analysis.

\subsection{Relaxation to the smooth case}

To follow the program above, we first perform an argument
in which (almost) every quantity is differentiable,
allowing us to essentially work with an unconstrained
problem, and for simplicity, we assume without loss of generality
(by scaling) that $a \in A_\statval(x)$ always satisfies
$\ltwo{a} \le 1$.
To that end, define the distance function
\begin{align*}
  \dfunc_X(x) = \inf_{y \in X} \ltwo{y - x}
  ~~ \mbox{with} ~~
  \half \nabla \dfunc_X^2(x) = x - \pi_X(x),
\end{align*}
where $\pi_X(x) = \argmin_{y \in X} \ltwo{y - x}$ denotes the Euclidean
projection onto $X$.
Because $\dfunc_X^2(x)$ is convex in $x$, we see that $x - \pi_X(x)$ is
maximal monotone (and even $1$-Lipschitz continuous)~\cite{BauschkeCo17}.
Consequently, for $\lambda > 0, \tau > 0$, we may define the (randomized)
continuous mapping
\begin{align}
  \label{eqn:regularized-T-mapping}
  T_{n,\lambda,\tau}(x) \defeq A_n(x) + \lambda(x - x_0)
  + \frac{1}{\tau} (x - \pi_X(x)).
\end{align}
Because $x \mapsto T_{n,\lambda,\tau}(x)$ is strongly (maximal) monotone
over $\R^d$, it has range $\R^d$ and is therefore a bijection between $\R^d$
and itself~\cite[Ch.~23]{BauschkeCo17}, and it has $(1/\lambda)$-Lipschitz
inverse
\begin{align*}
  T_{n,\lambda,\tau}^{-1}(z)
  \defeq \left\{x \mid A_n(x) + \lambda(x - x_0) - z + \frac{1}{\tau} (x - \pi_X(x)) = 0 \right\}.
\end{align*}

\begin{proposition}
  \label{proposition:smooth-monotone-laplace}
  Let $Z$ have generalized Laplace density~\eqref{eqn:generalized-Laplace}
  with $\alpha > 0$,
  and assume that
  $A_\statval$ are globally bounded,
  in that $\ltwo{a} \le b$ for
  all $a \in A_\statval(x)$ and $x \in \R^d$.
  Then
  \begin{equation*}
    \E\left[\ltwo{A_n(T_{n,\lambda,\tau}^{-1}(Z)) - A(T_{n,\lambda,\tau}^{-1}(Z))}
      \right]
    \le O(1) \cdot
    \left(\frac{b^2}{\alpha n} + b \sqrt{\frac{d}{n}}\right).
  \end{equation*}
\end{proposition}
\noindent
The proof of Proposition~\ref{proposition:smooth-monotone-laplace} is quite
involved, so we defer it to Section~\ref{sec:proof-smooth-monotone-laplace}.

Nonetheless, given Proposition~\ref{proposition:smooth-monotone-laplace},
the remainder of the steps become reasonably straightforward exercises
in convex analysis and integration.
For the coming argument, let $\delta_n$ be any value satisfying
\begin{align}
  \label{eqn:error-bound}
  \E[\ltwos{A_n(T_{n,\lambda,\tau}^{-1}(Z)) - 
      A(T_{n,\lambda,\tau}^{-1}(Z))}] \le \delta_n
\end{align}
uniformly in $A_\statval$ and the underlying
distribution $P$, so long as $A_\statval$ is uniformly bounded
and continuously differentiable.

\begin{enumerate}[leftmargin=*,label=(\arabic*)]
\item \label{item:tau-to-zero}
  We first take $\tau \downarrow 0$, which implies
  $T_{n,\lambda,\tau}^{-1}(z) \to x_n^\lambda(z)$ for each $z$,
  and dominated convergence then implies that
  \begin{align*}
    \E\left[\ltwo{A_n(x_n^\lambda(Z)) - A(x_n^\lambda(Z))}\right]
    \le \delta_n.
  \end{align*}
\item \label{item:normal-cone}
  We then address error in the normal cone term
  in $\dist(0, A(x) + \normalcone_X(x))$.
  By the condition
  \begin{align*}
    0 & \in A_n(x_n^\lambda)
    + \lambda(x_n^\lambda - x_0) - Z + \normalcone_X(x_n^\lambda) \\
    & = (A_n(x_n^\lambda) - A(x_n^\lambda))
    + A(x_n^\lambda) + \lambda (x_n^\lambda - x_0)
    - Z + \normalcone_X(x_n^\lambda)
    \end{align*}
  this scales with the error in $A_n(x_n^\lambda) - A(x_n^\lambda)$, so it
  is typically order $1/\sqrt{n}$.
\item \label{item:smoothing}
  For the final step,
  we perform randomized smoothing~\cite{DuchiBaWa12}; by approximation
  and outer semicontinuity, this will give the final result
  that
  \begin{align}
    \label{eqn:from-uniform-to-nonsmooth}
    \E\left[\dist(0, A(x_n^\lambda(Z)) +
      \normalcone_X(x_n^\lambda(Z)))\right]
    \le 2 \E[\ltwo{Z}^2]^{1/2}+ \delta_n + \frac{b}{\sqrt{n}}
    + 2 \lambda \ltwo{x_0 - x\opt}.
  \end{align}
\end{enumerate}
Substituting for $\delta_n$ in
inequality~\eqref{eqn:from-uniform-to-nonsmooth} using
the proposition then gives the theorem once we compute $\E[\ltwo{Z}]
= \alpha d$.

The remainder of this section implements the steps.

\subsection{Step~\ref{item:tau-to-zero}: removing the approximating boundary}

We first take $\tau \downarrow 0$ in the definition of $T_{n,\lambda,\tau}$.
The next fairly simple lemma
shows convergence of such approximate solutions:
\begin{lemma}
  Let $M$ be a bounded maximal monotone operator, $\lambda > 0$,
  and $x_\tau$ and $x_0$ solve
  \begin{align*}
    0 \in M(x_\tau) + \lambda x_\tau + \frac{1}{\tau} (x_\tau - \pi_X(x_\tau))
    ~~ \mbox{and} ~~
    0 \in M(x_0) + \lambda x_0 + \normalcone_X(x_0).
  \end{align*}
  Then $x_\tau \to x_0$ as $\tau \downarrow 0$, and
  each is unique.
\end{lemma}
\begin{proof}
  The operator $M_\lambda \defeq M + \lambda I_d$ is $\lambda$-strongly maximal
  monotone
  so that both $x_\tau$ and $x_0$ are unique~\cite{BauschkeCo17}.
  Abusing notation to let $M(x)$ be a selection of $M(x)$,
  we have
  \begin{align*}
    \lambda \ltwo{x_\tau - x_0}^2
    & \le \<M_\lambda(x_\tau) - M_\lambda(x_0), x_\tau - x_0\> \\
    & = \<M_\lambda(x_0), x_0 - x_\tau\>
    + \tau^{-1} \<\pi(x_\tau) - x_\tau, x_\tau - \pi(x_\tau)\>
    + \tau^{-1} \<\pi(x_\tau) - x_\tau, \pi(x_\tau) - x_0\> \\
    & \le \<M_\lambda(x_0), x_0 - x_\tau\>
    - \tau^{-1} \ltwo{x_\tau - \pi(x_\tau)}^2,
  \end{align*}
  where we use that $\<\pi(x_\tau) - x_\tau, \pi(x_\tau) - x_0\> \le 0$
  by definition of the projection.
  Rearranging and applying
  Cauchy-Schwarz yields that
  $\lambda \ltwo{x_\tau - x_0}^2 \le \ltwo{M_\lambda(x_0)} \ltwo{x_\tau - x_0}$,
  so that $\ltwo{x_\tau - x_0}$ is uniformly bounded as $\tau \downarrow 0$.
  The defining equality for $x_\tau$ and boundedness of $M$ thus guarantee
  that $\frac{1}{\tau} (x_\tau - \pi_X(x_\tau)) \in \normalcone_X(\pi(x_\tau))$
  remains bounded.
  Moving to a subsequence if necessary, the closed graph of maximal monotone
  operators, along with the outer semicontinuity of the normal cone
  mapping~\cite[Corollary 6.29]{RockafellarWe98} thus gives that $x_\tau \to
  \hat{x}$, where $\hat{x}$ satisfies
  \begin{align*}
    0 \in M(\hat{x}) + \lambda \hat{x} + \normalcone_X(\hat{x}).
  \end{align*}
  Strong monotonicity of $M_\lambda$ guarantees $\hat{x} = x_0$.
\end{proof}

Now, using the assumed continuity of $A_n$, we obtain that
$A_n(T_{n,\lambda,\tau}^{-1}(z))
- A(T_{n,\lambda,\tau}^{-1}(z)) \to
A_n(x_n^\lambda(z)) - A(x_n^\lambda(z))$ as $\tau \downarrow 0$
by the preceding lemma.
Dominated convergence implies that if inequality~\eqref{eqn:error-bound}
holds, then
\begin{equation}
  \lim_{\tau \downarrow 0}
  \E\left[\ltwo{A_n(T_{n,\lambda,\tau}^{-1}(Z))
      - A(T_{n,\lambda,\tau}^{-1}(Z))}\right]
  = \E\left[\ltwo{A_n(x_n^\lambda(Z))
      - A(x_n^\lambda(Z))}\right]
  \le \delta_n.
  \label{eqn:limit-smooth-result}
\end{equation}

\subsection{Step~\ref{item:normal-cone}: incorporating the normal cone}

By definition, we have
\begin{align*}
  0 & \in A_n(x_n^\lambda(z)) + \lambda(x_n^\lambda(z) - x_0)
  - z + \normalcone_X(x_n^\lambda(z)) \\
  & = A_n(x_n^\lambda(z))
  - A(x_n^\lambda(z))
  + A(x_n^\lambda(z))
  + \lambda(x_n^\lambda(z) - x_0)
  - z + \normalcone_X(x_n^\lambda(z)).
\end{align*}
In particular, there exists $w_n \in \normalcone_X(x_n^\lambda(z))$
for which
$A(x_n^\lambda(z)) + \lambda(x_n^\lambda(z) - x_0)
- z + w_n = A(x_n^\lambda(z)) - A_n(x_n^\lambda(z))$.
By definition, for any $x \in X$, vector $v$, and $w \in \normalcone_X(x)$
\begin{align*}
  \dist\left(0, v + \normalcone_X(x)\right)
  \le \ltwo{v + w},
\end{align*}
and so if inequality~\eqref{eqn:error-bound} holds, then by inequality~\eqref{eqn:limit-smooth-result},
\begin{align*}
  \E\left[\dist(0, A(x_n^\lambda(Z))
    + \lambda (x_n^\lambda(Z) - x_0) - Z
    + \normalcone_X(x_n^\lambda(Z)))\right]
  \le \delta_n.
\end{align*}
Using that $\dist(0, a + B) \ge \dist(0, B) - \ltwo{a}$ we obtain
\begin{align}
  \E\left[\dist(0, A(x_n^\lambda(Z))
    + \normalcone_X(x_n^\lambda(Z)))\right]
  \le \E[\ltwo{Z}] + \lambda \E[\ltwos{x_n^\lambda(Z) - x_0}]
  + \delta_n.
  \label{eqn:inter-smooth-error-normal-cone}
\end{align}

Lastly, we control the error
term $\ltwos{x_n^\lambda(Z) - x_0}$.
If $X$ is compact, we have the trivial
bound $\E[\ltwos{x_n^\lambda(Z) - x_0}] \le 2\radius(X)$.
Otherwise, the following lemma provides control:
\begin{lemma}
  \label{lemma:minimum-convergence}
  Let $x\opt$ solve problem~\eqref{eqn:vi}.
  Then
  \begin{equation*}
    \E\left[\ltwo{x_n^\lambda(Z) - x\opt}^2\right]^{1/2}
    \le \frac{b}{\lambda \sqrt{n}}
    + \ltwo{x\opt - x_0}
    + \frac{\E[\ltwo{Z}^2]^{1/2}}{\lambda}.
  \end{equation*}
\end{lemma}
\begin{proof}
  Use the shorthand $x_n = x_n^\lambda(Z)$, and
  let $a_n(x) \in A_n(x)$ be a selection from $A_n$, and in the case of
  $x_n$, let it be the selection guaranteeing $0 \in a_n(x_n) + \lambda(x_n
  - x_0) - z + \normalcone_X(x_n)$, that is, $\<a_n(x_n) + \lambda(x_n -
  x_0) - z, y - x_n\> \ge 0$ for all $y \in X$.
  Then by (strong) monotonicity of $x \mapsto A_n(x) + \lambda (x - x_0)$, we
  have
  \begin{align*}
    \lambda \ltwo{x_n - x\opt}^2
    & \le \<a_n(x_n) - a_n(x\opt) + \lambda (x_n - x\opt), x_n - x\opt\> \\
    & = \<a_n(x_n) + \lambda(x_n - x_0) - z
    - (a_n(x\opt) + \lambda(x\opt - x_0) - z), x_n - x\opt\> \\
    & \le \<a_n(x\opt) + \lambda(x\opt - x_0) - z, x\opt - x_n\>.
  \end{align*}
  Let $w \in \normalcone_X(x\opt)$ satisfy
  $0 = \E[a_n(x\opt) + w]$,
  noting that $\<w, x_n - x\opt\> \le 0$.
  Then adding
  and subtracting $\<w, x\opt - x_n\> \ge 0$ gives
  \begin{align*}
    \lambda \ltwo{x_n - x\opt}^2
    \le \<a_n(x\opt) + \lambda(x\opt - x_0) + w - z, x\opt - x_n\>.
  \end{align*}
  By Cauchy-Schwarz,
  $\E[\<a_n(x\opt) + w, x\opt - x_n\>]
  \le \E[\ltwo{a_n(x\opt) + w}^2]^{1/2} \E[\ltwo{x\opt - x_n}^2]^{1/2}$,
  and under the boundedness assumptions
  we have $\E[\ltwo{a_n(x\opt) + w}^2] \le b^2 / n$
  because $\E[a_n(x\opt) + w] = 0$.
  So
  \begin{align*}
    \lambda\E\left[\ltwos{x_n -x\opt}^2\right]
    & \le \frac{b}{\sqrt{n}} \E[\ltwo{x_n - x\opt}^2]^{1/2}
    + \left(\lambda \ltwo{x\opt - x_0} + \E[\ltwo{Z}^2]^{1/2}
    \right) \E[\ltwo{x_n - x\opt}^2]^{1/2}.
  \end{align*}
  Divide by $\E[\ltwos{x_n - x\opt}^2]^{1/2}$.
\end{proof}

\subsection{Step~\ref{item:smoothing}: smoothing the operators}

\newcommand{\res}[2]{\proxmap_{#1 #2}}
\newcommand{\yosida}[2]{Y_{#1 #2}}

When the operators $A_\statval$ need not be single-valued,
the preceding bounds do not apply.
Nonetheless, because the graph of maximal monotone operators is closed (and
hence $x \mapsto A_\statval(x)$ is
outer-semicontinuous)~\cite[Prop.~20.37]{BauschkeCo17}, we may approximate
them by smooth operators.
Thus, we define the resolvent (proximal) and Yosida
approximations~\cite[Ch.~23]{BauschkeCo17}
\begin{align*}
  \res{\epsilon}{A} \defeq (I + \epsilon A)^{-1}
  ~~~ \mbox{and} ~~~
  \yosida{\epsilon}{A} \defeq \frac{1}{\epsilon} (I - \res{\epsilon}{A})
  = (\epsilon I + A^{-1})^{-1}
\end{align*}
for $\epsilon > 0$.
The resolvent of $A$ and the Yosida approximation to $A$ are both
Lipschitz continuous and monotone, so that $\yosida{\epsilon}{A}$ is
maximal monotone~\cite[Corollary 20.28]{BauschkeCo17},
and moreover,
\begin{align*}
  \yosida{\epsilon}{A}(x) \in A(\res{\epsilon}{A}(x)),
  ~~ \mbox{so} ~~
  \ltwo{\yosida{\epsilon}{A}(x)}
  \le \sup\left\{\ltwo{a}  \mid a \in A(\res{\epsilon}{A}(x))\right\}
\end{align*}
for any $x$.
Let $\rho$ be any $\mc{C}^\infty$ compactly supported density.
For $\epsilon > 0$ and monotone $A$, we can consider the smoothed operator
\begin{align*}
  \yosida{\epsilon}{A}^\rho(x)
  \defeq \int \yosida{\epsilon}{A}(x - u) \epsilon^{-d}
  \rho(u / \epsilon) du,
\end{align*}
which is $\mc{C}^\infty$ (because $\rho$ is) and bounded by $\sup_y
\{\ltwo{A(y)}\}$.
The outer-semicontinuity of $x \mapsto A_\statval(x)$ shows that if
$\epsilon_k \to 0$ and $x_k \to x$, then
\begin{align*}
  \limsup_{k \to \infty}
  \yosida{\epsilon_k}{A_\statval}^\rho(x_k) \subset A_\statval(x)
\end{align*}
for each $\statval \in \statdomain$.
Dominated convergence then implies that, as set-valued mappings,
\begin{align*}
  \limsup_{k \to \infty}
  \left\{\E_P[\yosida{\epsilon_k}{A_\statrv}^\rho(x_k)]
  = \int \yosida{\epsilon_k}{A_\statval}^\rho(x_k) dP(\statval)
  \right\} \subset A(x).
\end{align*}

For $\epsilon > 0$, let $y_\epsilon(z)$ solve
\begin{align*}
  0 \in \frac{1}{n} \sum_{i = 1}^n \yosida{\epsilon}{A_{\statrv_i}}^\rho(x)
  + \lambda (x - x_0) - z + \normalcone_X(x).
\end{align*}
Then $y_\epsilon(z)$ has convergent subsequences as $\epsilon \downarrow 0$,
because
$x \mapsto \lambda (x - x_0)$ is a strongly monotone
operator~\cite[Ch.~23]{BauschkeCo17} and so $y_\epsilon(z)$ remains bounded
for all $\epsilon \ge 0$;
let $\hat{x}$ be one of these limits.
Proceeding to a further subsequence if necessary,
outer semicontinuity demonstrates that
\begin{align*}
  0 \in A_n(\hat{x}) + \lambda(\hat{x} - x_0) - z + \normalcone_X(\hat{x}),
\end{align*}
and because $\lambda > 0$ so that
$x_n^\lambda(z)$ is unique for each $z$, we have
$y_\epsilon(z) \to x_n^\lambda(z)$ as $\epsilon \downarrow 0$.
Now, we let
$g_\epsilon(z) = \E_P[\yosida{\epsilon}{A_\statrv}^\rho(y_\epsilon(z))]$,
and choose $w_\epsilon(z) \in \normalcone_X(y_\epsilon(z))$
realizing the distance
\begin{equation*}
  \ltwo{g_\epsilon(z) + w_\epsilon(z)}
  = \dist(0, g_\epsilon(z) + \normalcone_X(y_\epsilon(z))).
\end{equation*}
Then the convergence of $y_\epsilon(z) \to x_n^\lambda(z)$ for each $z$
and outer semicontinuity yield vectors
$g_0 \in A(x_n^\lambda(z))$ and $w_0 \in \normalcone_X(x_n^\lambda(z))$ for which
\begin{align*}
  \liminf_{\epsilon \downarrow 0}
  \ltwo{g_\epsilon(z) + w_\epsilon(z)}
  = \ltwo{g_0(z) + w_0(z)} \ge
  \dist\left(0, A(x_n^\lambda(z)) + \normalcone_X(x_n^\lambda(z))\right).
\end{align*}
Fatou's lemma implies
\begin{align*}
  \E\left[\dist(0, A(x_n^\lambda(Z))
    + \normalcone_X(x_n^\lambda(Z))) \mid P_n\right]
  & \le \int \liminf_{\epsilon \downarrow 0}
  \ltwo{g_\epsilon(z) + w_\epsilon(z)} q(z) dz \\
  & \le \liminf_{\epsilon \downarrow 0} \int \ltwo{g_\epsilon(z) + w_\epsilon(z)}
  q(z) dz.
\end{align*}
Applying Fatou's lemma once again
by taking the expectation over the sample $P_n$,
\begin{align*}
  \lefteqn{\E\left[\dist(0, A(x_n^\lambda(Z))
      + \normalcone_X(x_n^\lambda(Z)))\right]} \\
  & \le \liminf_{\epsilon \downarrow 0}
  \E\left[\ltwo{g_\epsilon(Z) + w_\epsilon(Z)}\right]
  \le \E[\ltwo{Z}^2]^{1/2} + \delta_n
  + \lambda \liminf_{\epsilon \downarrow 0}
  \E\left[\ltwo{y_\epsilon(Z) - x_0}\right]
\end{align*}
by the smooth case in inequality~\eqref{eqn:inter-smooth-error-normal-cone}
applied at any $\epsilon > 0$.

To bound the final limit infimum, note that $y_\epsilon(z) \to x_n^\lambda(z)$
for each $z$, and because
$\ltwos{\yosida{\epsilon}{A_\statval}^\rho(x)} \le b$ for each $\statval$,
we have $\ltwo{y_\epsilon(z) - x_0} \le \lambda^{-1}(\ltwo{z} + b)$.
Dominated convergence then implies $\E[\ltwo{y_\epsilon(Z) - x_0}]
\to \E[\ltwo{x_n^\lambda(Z) - x_0}]$,
and so Lemma~\ref{lemma:minimum-convergence}
implies
\begin{align}
  \nonumber
  \lefteqn{\E\left[\dist(0, A(x_n^\lambda(Z))
    + \normalcone_X(x_n^\lambda(Z)))\right]} \\
  & \le \E[\ltwo{Z}] + \delta_n
  + \min\left\{2\lambda \cdot \radius(X),
  \E[\ltwo{Z}^2]^{1/2} + \frac{b}{\sqrt{n}}
  + 2 \lambda \ltwo{x_0 - x\opt}\right\} \nonumber \\
  & \le 2 \E[\ltwo{Z}^2]^{1/2} + \delta_n
  + 2 \lambda \ltwo{x_0 - x\opt}
  + \frac{b}{\sqrt{n}},
  \label{eqn:smooth-error-normal-cone}
\end{align}
by Cauchy-Schwarz, as $\E[\ltwo{Z}]\le \E[\ltwo{Z}^2]^{1/2}$.


\section{Proof of Proposition~\ref{proposition:smooth-monotone-laplace}}
\label{sec:proof-smooth-monotone-laplace}

By scaling, it will be no loss of generality to assume
that the operators $A_\statval$ are such that
$\ltwo{A_\statval(x)} \le \half$ for each $x, \statval$,
which in turn implies that
$\ltwo{A_\statval(x) - A(x)} \le 1$.
We follow the rough program outlined in the beginning of
Section~\ref{sec:monotone-concentration}.
In this case, we assume $A_\statval$ is $\mc{C}^1$, so that
it has derivative $\dot{A}_\statval(x)$ everywhere.
Because $\dfunc_X^2(x)$ is convex with Lipschitz continuous gradient, Rockafellar's
refinement of Alexandrov's theorem~\cite{Rockafellar99} guarantees that
$\nabla^2 \dfunc_X^2(x)$ exists almost everywhere, and it is symmetric
positive semidefinite.
Thus, defining the regularization part of the matrix
\begin{align*}
  B_{\lambda, \tau}(x)
  \defeq \lambda I_d + \frac{1}{2 \tau} \nabla^2 \dfunc_X^2(x),
\end{align*}
the continuous mapping~\eqref{eqn:regularized-T-mapping}
has derivative
\begin{align*}
  \deriv T_{n,\lambda,\tau}(x)
  = \frac{1}{n} \sum_{i = 1}^n \dot{A}_{\statrv_i}(x)
  + B_{\lambda, \tau}(x)
\end{align*}
for almost all $x \in \R^d$.
Thus, by the area formula~\cite[Thm.~3.9]{EvansGi15},
we have
\begin{align}
  \nonumber
  \lefteqn{\E\left[\ltwo{A_n((T_{n,\lambda,\tau})^{-1}(Z))
        - A(T_{n,\lambda,\tau}^{-1}(Z))} \mid P_n\right]} \\
  & = \int_{\R^d} \ltwo{A_n((T_{n,\lambda,\tau})^{-1}(z))
    - A(T_{n,\lambda,\tau}^{-1}(z))} q(z) dz \nonumber \\
  & = \int_{\R^d}
  q(T_{n,\lambda,\tau}(x)) \det(\deriv T_{n,\lambda,\tau}(x))
  \ltwo{A_n(x) - A(x)} dx
  \label{eqn:density-from-determinant}
\end{align}
for any density $q$.
The representation~\eqref{eqn:density-from-determinant}
provides the key to progress: because $A_n - A$ concentrates to 0
and the other integrands integrate to 1,
we can show it decreases quickly in $n$.
The remainder of the proof proceeds as follows: because $B \mapsto \det(B)$
is a $d$th-order homogeneous polynomial in $B$, we may essentially write
$\det (\deriv T_{n,\lambda,\tau}(x))$ as a sum over functions involving at
most $d$ observations $\statrv_i$ at a time, leaving the remaining $n - d$
random.
This decomposition follows from the Hoeffding
decomposition~\cite[Ch.~11]{VanDerVaart98} familiar from expansions
used to control moments of $U$-statistics.
Conditioning on those, we may argue that $\ltwo{A_n(x) - A(x)}$ is still small.
(We remark that this type of idea appears in the literature on sampling
determinantal-point-processes for empirical risk
minimization~\citet[cf.][]{DerezinskiWaHs22}, though they use the
Cauchy-Binet formula rather than the expansion we develop.)

\subsection{Expectations of limited interacting terms}

We first develop technical results related to expectations of random
variables that interact in ``limited'' ways.
For a collection of random variables $Y_i$, $i = 1, \ldots, n$, let $F$ be a
(potentially vector-valued) function of $Y_1^n$, and consider its Hoeffding
decomposition~\cite[Ch.~11.4]{VanDerVaart98}.
Following~\citet[Lemma 11.11]{VanDerVaart98}, we can write any such function
$F$ as a sum of its orthogonal $L^2$-projections onto functions of subsets
of the random variables $Y_1^n$; for $J \subset [n]$ let $\mc{H}_J$ denote
the Hilbert space of functions of the form $g(Y_J) = g(Y_i : i \in J)$ where
$\E[\ltwo{g}^2] < \infty$ and $\E[g(Y_J) \mid Y_K] = 0$ whenever $|K| <
|J|$.
Then using the empirical process notation and letting $P_i$ be the operator
$(P_i f) (y_{\setminus i}) = \int f(y_1^n) dP(y_i) = \E[f \mid Y_{\setminus
    i} = y_{\setminus i}]$, that is, conditional expectation integrating
only $Y_i$, we can define
\begin{align*}
  F_J = \left(\prod_{i \in J} (\id - P_i)
  \prod_{i \in J^c} P_i \right) F
  = \sum_{K \subset J} (-1)^{|J| - |K|}
  \E[F \mid Y_K],
\end{align*}
from which we have~\citep[Lemma 11.11]{VanDerVaart98}
\begin{align}
  \label{eqn:hoeffding-decomposition}
  F(Y_1^n) = \sum_{J \subset [n]} F_J(Y_J).
\end{align}
Notably, $\E[F_J \mid Y_K] = 0$ whenever $|K| < |J|$, and
$F_\emptyset = \E[F]$.
We say that $F$ has \emph{interaction order at most $k$} if $F_J = 0$
whenever $|J| > k$; for example, any polynomial of degree at most $k$
(meaning the total degree of each of its terms is $\le k$) has
interaction order at most $k$.

Considering the area formula~\eqref{eqn:density-from-determinant},
it is clear that we will need to integrate various means against
other functions of (potentially) limited interaction order.
We have the following technical estimate, whose proof
we defer to Appendix~\ref{sec:proof-interaction-order-moments}.
\begin{lemma}
  \label{lemma:interaction-order-moments}
  Let $Y_i$ be mean-zero independent random variables with
  $\ltwo{Y_i} \le 1$
  and define the sum $S_n \defeq \sum_{i = 1}^n Y_i$.
  Assume that $F$ has interaction order at most $k$.
  Then there is a numerical constant $c$, independent of $F$, $k$,
  and the distribution of $Y_i$, such that
  \begin{align*}
    \E\left[\ltwo{S_n}^2 \ltwo{F(Y_1^n)}^2 \right]
    \le c \cdot n (k + 1) \cdot \E[\ltwo{F(Y_1^n)}^2].
  \end{align*}
\end{lemma}

\subsection{Expansions of the determinant using the Pfaffian}

Recalling
the area formula~\eqref{eqn:density-from-determinant},
to apply Lemma~\ref{lemma:interaction-order-moments}
we will need to write the determinant
$\det(\deriv T_{n,\lambda,\tau})$ as a sum of squared terms of limited
interaction.
At some level, this is immediate:
$\det$ is a $d$th order homogeneous polynomial, so the higher order
terms in its Hoeffding decomposition~\eqref{eqn:hoeffding-decomposition}
must be zero.
However, because we seek upper bounds on the
integral~\eqref{eqn:density-from-determinant}, we must write the determinant
as a sum of squares to apply Lemma~\ref{lemma:interaction-order-moments}.

To that end,
recall that determinant of a $2m \times 2m$-skew-symmetric matrix
$B$ can be computed via its Pfaffian,
where
\begin{align}
  \label{eqn:pfaffian-definition}
  \pf(B) = \frac{1}{2^m m!} \sum_{\sigma \in \mathfrak{S}_{2m}}
  \sgn(\sigma) \prod_{i = 1}^m b_{\sigma(2i - 1), \sigma(2i)}
\end{align}
and we have
\begin{align*}
  \pf(B)^2 = \det(B).
\end{align*}
This identification allows us to construct the determinant of a skew
symmetric plus positive definite matrix explicitly
as a sum of degree $O(d)$ polynomials, which will be essential
for our development.
\begin{lemma}
  \label{lemma:det-as-pfaffian}
  Let $\mc{J}_d$ be the collection of subsets
  $J \subset [n]$ with $|J| \le d$ and $d + |J|$ even.
  Let $B \in \R^{d \times d}$ be skew-symmetric and $V \in \R^{d \times n}$,
  and let $V_J \in \R^{d \times |J|}$ indicate the submatrix of $V$ with
  columns indexed by $J$.
  Then
  \begin{align*}
    \det(B + VV^\top) = \sum_{J \in \mc{J}_d}
    \pf\left(
    \begin{matrix} B & V_J \\ -V_J^\top & 0 \end{matrix}
    \right)^2.
  \end{align*}
  Moreover, for each $J \in \mc{J}_d$ and fixed $V_J$, the Pfaffian terms
  (without square) are each $(d - |J|)/2$th order homogeneous polynomials in
  the entries of $B$.
\end{lemma}
\noindent
Once again, this lemma is a technical identity; we defer its
proof to Appendix~\ref{sec:proof-det-as-pfaffian}.

\subsection{Decomposing into subsamples}

By assumption in this argument, the operators $A_\statval$ are continuously
differentiable with derivatives $\dot{A}_\statval(x) \in \R^{d \times d}$,
where maximal monotonicity implies the
decomposition of $\dot{A}_\statval$ into a symmetric positive semidefinite
component $H_\statval$
and skew-symmetric component $M_\statval$
via $\dot{A}_\statval(x) = H_\statval(x)
+ M_\statval(x)$, where explicitly
\begin{align*}
  H_\statval(x) = \half(\dot{A}_\statval(x) + \dot{A}_\statval(x)^\top)
  \succeq 0
  ~~ \mbox{and} ~~
  M_\statval(x)
  = \half (\dot{A}_\statval(x) - \dot{A}_\statval(x)^\top).
\end{align*}
(In the case that $A_\statval = \nabla F_\statval$ is the gradient
of a convex function, then $H_\statval = \nabla^2 F_\statval$
and $M_\statval = 0$.)

We rewrite the derivative~\eqref{eqn:regularized-T-mapping}
in the form of Lemma~\ref{lemma:det-as-pfaffian}.
Recalling
$T_{n,\lambda,\tau}(x) = A_n(x) + \lambda(x - x_0)
+ \frac{1}{\tau} (x - \pi_X(x))$
and $B_{\lambda,\tau}(x) = \lambda I_d + \frac{1}{2 \tau}
\nabla^2 \dfunc_X^2(x) \succeq 0$ (when it exists).
By defining the matrix square roots
$V_i(x) = \frac{1}{\sqrt{n}} H_{\statrv_i}(x)^{1/2}$, we may write
\begin{align*}
  \deriv T_{n,\lambda,\tau}(x)
  = \frac{1}{n} \sum_{i = 1}^n M_{\statrv_i}(x)
  + \left[B_{\lambda,\tau}(x)^{1/2}
    ~ V_1(x) ~ \cdots ~ V_n(x) \right]
  \left[B_{\lambda,\tau}(x)^{1/2}
    ~ V_1(x) ~ \cdots ~ V_n(x) \right]^\top,
\end{align*}
which follows the form of Lemma~\ref{lemma:det-as-pfaffian}.
For $|J| \le d$,
let $V_J(x) \in \R^{d \times |J|}$ denote the sub-matrix
consisting of columns $i \in J$ of
\begin{align*}
  V(x) \defeq \left[B_{\lambda,\tau}(x)^{1/2}
    ~~ n^{-1/2} H_{\statrv_1}(x)^{1/2} ~~ \cdots ~~ n^{-1/2} H_{\statrv_n}^{1/2}(x)
    \right]
  \in \R^{d \times d(n+1)}.
\end{align*}
Applying Lemma~\ref{lemma:det-as-pfaffian} thus yields
\begin{align*}
  \det(\deriv T_{n,\lambda,\tau}(x))
  & = \sum_{J \in \mc{J}_d}
  F_J^2(x \mid P_n)
  ~~~ \mbox{where} ~~~
  F_J(x \mid P_n)
  = \pf\left(\begin{matrix} M_n(x) &
    V_J(x) \\ -V_J(x)^\top & 0 \end{matrix}\right).
\end{align*}
For a subset $J$, let
$\hat{J}$ denote the sample indices in $\{1, \ldots, n\}$ corresponding to
columns of $V$ that $J$ selects, so that $|\hat{J}| \le d$.
Then $V_J(x)$ is $\statrv_{\hat{J}}$-measurable,
and as the Pfaffian term $F_J(x \mid P_n)$
is an at most $\floor{d/2}$-degree homogeneous polynomial
in $M_n(x)$,
it has interaction order at most $\floor{d/2}$ conditional
on $\statrv_{\hat{J}}$.
(Recall the Hoeffding decomposition~\eqref{eqn:hoeffding-decomposition} and
associated definition of interaction order.)
We record this as a lemma:
\begin{lemma}
  \label{lemma:decomposition-det-operator-case}
  For Lebesgue-almost-all $x$,
  we have
  \begin{align*}
    \det(\deriv T_{n,\lambda,\tau}(x))
    = \sum_{J \in \mc{J}_d} F_J^2(x \mid P_n),
  \end{align*}
  where for each $J \in \mc{J}_d$, there is a collection
  of indices $\hat{J} \subset [n]$ with $|\hat{J}| \le d$ for which
  conditional on $\statrv_{\hat{J}}$,
  $F_J(x \mid P_n)$ has interaction order at most
  $\floor{d/2}$.
\end{lemma}

From Lemma~\ref{lemma:decomposition-det-operator-case},
we see that the area formula~\eqref{eqn:density-from-determinant}
thus becomes
\begin{align}
  \nonumber
  \lefteqn{\E\left[\ltwo{A_n(T_{n,\lambda,\tau}^{-1}(Z))
        - A(T_{n,\lambda,\tau}^{-1}(Z))} \mid P_n\right]} \\
  & =
  \sum_{J \in \mc{J}_d} \int_{\R^d}
  F_J^2(x \mid P_n) q(T_{n,\lambda,\tau}(x))
  \ltwo{A_n(x) - A(x)} dx.
  \label{eqn:area-formula-revisited}
\end{align}
Consider the expectation of the rightmost integrand over the
sample $\statrv_1^n$: for a set $J \in \mc{J}_d$ with
associated sample indices $\hat{J} \subset [n]$, $|\hat{J}| \le |J|$,
we have
\begin{align*}
  \E\left[F_J^2(x \mid P_n) q(T_{n,\lambda,\tau}(x))
    \ltwo{A_n(x) - A(x)}\right]
  & = \E\left[\E\left[F_J^2(x \mid P_n) q(T_{n,\lambda,\tau}(x))
      \ltwo{A_n(x) - A(x)} \mid \statrv_{\hat{J}} \right]\right].
\end{align*}
Conditional on $\statrv_{\hat{J}}$,
the Pfaffian $F_J(x \mid P_n)$ has interaction order at most
$\floor{d/2}$, and we can expand
\begin{subequations}
  \label{eqn:expand-in-terms-of-J}
  \begin{align}
  A_n(x) - A(x)
  = \frac{1}{n} \sum_{i \not \in \hat{J}} (A_{\statrv_i}(x) - A(x))
  + \frac{1}{n} \sum_{i \in \hat{J}} (A_{\statrv_i}(x) - A(x))
  \end{align}
  and
  \begin{equation}
    \begin{split}
      T_{n,\lambda,\tau}(x)
      & = (A_n(x) - A(x)) + A(x) + \lambda (x - x_0)
      + \frac{1}{\tau} (x - \pi_X(x)) \\
      & = \frac{1}{n} \sum_{i \not \in \hat{J}}
      \left(A_{\statrv_i}(x) - A(x)\right)
      + \frac{1}{n} \sum_{i \in \hat{J}}
      \left(A_{\statrv_i}(x) - A(x)\right)
      + A(x)
      + \lambda (x - x_0) + \frac{1}{\tau} (x - \pi_X(x))
    \end{split}
  \end{equation}
\end{subequations}
As such, after appropriate re-indexing, we seek generically to bound
quantities of the following form: let $Y_i$ be independent mean-zero vectors
with $\ltwo{Y_i} \le 1$, and let $\wb{Y}_{n,m} = \frac{1}{n} \sum_{i = 1}^m
Y_i$.
Then for arbitrary vectors $w$ and $v$, we wish to control
\begin{align}
  \label{eqn:F-with-densities}
  \E\left[F^2 q(\wb{Y}_{n,m} + v) \ltwo{\wb{Y}_{n,m} + w}
    \right],
\end{align}
where $F$ has interaction order at most $d$.
We will use the Bobkov-Ledoux entropy method to show, with $q$ taken to
be the generalized Laplace density~\eqref{eqn:generalized-Laplace}, 
that the quantity~\eqref{eqn:F-with-densities} often
satisfies clean bounds.
Because the triangle inequality implies
$\ltwo{\wb{Y}_{n,m} + w} \le \ltwo{\wb{Y}_{n,m}} + \ltwo{w}$, we shall
typically let $w = 0$ and incorporate it later in the argument.

\subsection{The entropy method and density bounds}

To control quantities of the form~\eqref{eqn:F-with-densities}, we use the
Bobkov-Ledoux entropy method~\cite{Ledoux01, BoucheronLuMa13} via the
Donsker-Varadhan variational inequality.
To employ the entropy method, we will typically require some type
of sub-Gaussianity; the next two lemmas
demonstrate that the norms of sums
$\sqrt{n}\ltwo{\wb{Y}_n}$ still exhibit sub-Gaussian behavior.

\begin{lemma}
  \label{lemma:induct-out-interactions}
  There is a numerical constant $C < \infty$ such that the
  following holds.
  Let $Y_i$ be mean-zero independent variables with $\ltwo{Y_i} \le 1$ and $F$ have
  interaction order at most $k$, where $\E[F(Y_1^n)^2] < \infty$.
  Then for all $r \in \N$,
  \begin{align*}
    \E\left[\ltwo{\wb{Y}_n}^{2r} F^2(Y_1^n)\right]
    \le \left(\frac{C}{n}\right)^r \frac{(k + r)!}{k!} \E[F^2(Y_1^n)].
  \end{align*}
\end{lemma}
\begin{proof}
  For vectors $u \in \R^d$, $v$, use the tensor product
  $u \otimes v = (u_1 v, \ldots, u_d v)$ with
  Euclidean norm $\ltwo{u \otimes v} = \ltwo{u} \ltwo{v}$.
  Then
  recursively define the random variables
  $W_0 = F$, $W_1 = \wb{Y}_n \otimes F
  = ([\wb{Y}_n]_1 F, \ldots, [\wb{Y}_n]_d F) = \wb{Y}_n \otimes W_0$,
  through $W_r = \wb{Y}_n \otimes W_{r-1}$.
  Then by inspection,
  \begin{align*}
    \ltwo{W_r} = \ltwo{\wb{Y}_n} \ltwo{W_{r-1}}
    = \cdots = \ltwo{\wb{Y}_n}^r |F|,
  \end{align*}
  and the tensor product structure means
  the $l$th variable $W_l$ in the sequence
  has interaction degree at most $k + l$.
  Thus Lemma~\ref{lemma:interaction-order-moments} implies
  \begin{align*}
    \E\left[\ltwo{W_r}^2\right]
    = \E\left[\ltwo{\wb{Y}_n}^{2r} F^2(Y_1^n)\right]
    & = \E\left[\ltwo{\wb{Y}_n}^2 \ltwo{W_{r-1}}^2\right]
    \le c \frac{k + r}{n} \E\left[\ltwo{W_{r-1}}^2\right],
  \end{align*}
  where $c \le 8$ is a numerical constant.
  Applying the obvious induction gives the result.
\end{proof}

\begin{lemma}
  \label{lemma:exponential-squared-gaussian}
  Assume that $\ltwo{Y_i} \le 1$ for each $i$ and that $Y_i$ are mean-zero and independent.
  There exists a numerical constant $c < \infty$ such that
  for all $k \in \N$ and $F$ with interaction order at most $k$,
  \begin{align*}
    \E\left[\exp\left(\frac{n \ltwo{\wb{Y_n}}^2}{c}\right) F^2(Y_1^n)\right]
    \le 2^{k + 1} \E[F^2(Y_1^n)].
  \end{align*}
\end{lemma}
\begin{proof}
  For any $c > 0$, expanding the exponential
  we have
  \begin{align*}
    \exp\left(\frac{n y^2}{c}\right)
    = \sum_{r = 0}^\infty \left(\frac{n}{c}\right)^r
    \frac{y^{2r}}{r!},
  \end{align*}
  so Lemma~\ref{lemma:induct-out-interactions}
  yields that for a numerical constant $C < \infty$,
  \begin{align*}
    \E\left[\exp\left(\frac{n \ltwos{\wb{Y}_n}^2}{c}\right)
      F^2(Y_1^n)\right]
    & \le \sum_{r = 0}^\infty
    \left(\frac{C}{c}\right)^r
    \binom{k + r}{r} \E[F^2(Y_1^n)].
  \end{align*}
  Because the negative binomial generating function
  satisfies $\sum_{r = 0}^\infty t^r \binom{k + r}{r} = (1-t)^{-k - 1}$
  for $|t| < 1$, we take
  $c = 2 \cdot C$ and observe
  that $\sum_{r = 0}^\infty 2^{-r} \binom{k + r}{r}
  = 2^{k + 1}$.
\end{proof}

Recall that one of the equivalent formulations
of the $\sigma^2$-sub-Gaussianity of random variable $W$ is that
$\E[\exp(W^2 / \sigma^2)] \le e$; see, for example, \citet{Vershynin19}.
Then Lemma~\ref{lemma:exponential-squared-gaussian} shows that
if $F$ has interaction order $k$ and
$Y_i$ independent, mean-zero, and satisfy $\ltwo{Y_i} \le 1$,
then
$\sqrt{n} \ltwo{\wb{Y}_n}$ remains sub-Gaussian (of order $O(k)$)
even under the tilted measure
\begin{align*}
  \frac{d\nu}{dP}(y_1^n) = \frac{F^2(y_1^n)}{\E_P[F^2(Y_1^n)]},
\end{align*}
where $P$ is the initial product measure on the $Y_i$.
Thus, for any measure $\mu$ on $Y_1^n$, the Donsker-Varadhan
variational representation of the KL-divergence that
$\E_\mu[h] \le \dkl{\mu}{\nu} + \log \E_\nu[e^h]$ for any measurable
$h$ implies
\begin{align}
  \label{eqn:transfer-F-measure}
  c n \cdot \E_\mu\left[\ltwo{\wb{Y}_n}^2\right]
  & \le \dkl{\mu}{\nu} + (k + 1) \log 2,
\end{align}
where $c > 0$ is a numerical constant.
As a second consequence of Lemma~\ref{lemma:exponential-squared-gaussian},
we see via Jensen's inequality that
\begin{align*}
  \frac{1}{c} \E_\nu\left[n \ltwo{\wb{Y}_n}^2\right]
  \le
  \log\E_\nu\left[\exp(c^{-1} n \ltwo{\wb{Y}_n}^2)\right]
  \le (k + 1) \log 2,
\end{align*}
that is,
\begin{equation}
  \label{eqn:nu-means-small}
  \E_\nu\left[\ltwo{\wb{Y}_n}^2\right]
  \le O(1) \frac{k + 1}{n}.
\end{equation}

\paragraph{Specific control for Laplacian noise addition}

As our final direct probabilistic calculation for the proof of
Proposition~\ref{proposition:smooth-monotone-laplace}, we give explicit
calculations for quantities of the form~\eqref{eqn:transfer-F-measure} when
$q$ has the generalized Laplace density~\eqref{eqn:generalized-Laplace}.

\begin{lemma}
  \label{lemma:laplace-F-transfer}
  Let $F$ have interaction order at most $d \ge 1$
  and assume that $Y_i$ are independent and mean-zero with
  $\ltwo{Y_i} \le b$ for some $b < \infty$.
  Then
  \begin{align*}
    \E_P\left[\ltwo{\wb{Y}_n} F^2(Y_1^n)
      q(\wb{Y}_n + v)\right]
    \le O(1) \left[\frac{b^2}{\alpha n}
      + b \sqrt{\frac{d}{n}}\right] \E_P\left[F^2(Y_1^n) q(\wb{Y}_n + v)
      \right].
  \end{align*}
\end{lemma}
\begin{proof}
  Because we wish to control expectations of products $\ltwo{\wb{Y}_n} F^2
  q(\wb{Y}_n + v)$, we begin with the
  inequality~\eqref{eqn:transfer-F-measure}, and define the measure
  \begin{align*}
    \frac{d\mu}{d\nu} (y_1^n) = \frac{q_\alpha(\wb{y} + v)}{
      \E_\nu[q_\alpha(\wb{Y}_n + v)]}
    = e^{-\ltwo{\wb{y} + v} / \alpha} / \E_\nu[e^{-\ltwo{\wb{Y}_n + v}/\alpha}].
  \end{align*}
  Then a direct calculation shows
  \begin{align*}
    \lefteqn{\dkl{\mu}{\nu}} \\
    & = -\frac{1}{\alpha} \E_\mu\left[\ltwo{\wb{Y}_n + v}\right]
    - \log \E_\nu[\exp(-\ltwo{\wb{Y}_n + v} / \alpha)]
    \le \frac{1}{\alpha}
    \left(\E_\nu\left[\ltwo{\wb{Y}_n + v}\right]
    - \E_\mu\left[\ltwo{\wb{Y}_n + v}\right]\right).
  \end{align*}
  By Jensen's inequality and the triangle inequality, we have the trivial
  bounds
  \begin{align*}
    \E_\mu\left[\ltwo{\wb{Y}_n + v}\right]
    \ge \ltwo{\E_\mu[\wb{Y}_n] + v}
    \ge \ltwo{v} - \E_\mu\left[\ltwo{\wb{Y}_n}\right],
  \end{align*}
  and so applying inequality~\eqref{eqn:nu-means-small},
  \begin{align*}
    \dkl{\mu}{\nu}
    \le \frac{1}{\alpha}
    \left(\E_\nu\left[\ltwo{\wb{Y}_n}\right]
    + \E_\mu\left[\ltwo{\wb{Y}_n}\right]\right)
    \le \frac{1}{\alpha} \left(O(1) b \sqrt{\frac{d}{n}}
    + \E_\mu[\ltwo{\wb{Y}_n}]\right).
  \end{align*}
  Substituting in the bound~\eqref{eqn:transfer-F-measure},
  substituting $n/b^2$ for $n$ on the left side to address
  the boundedness,
  and applying Jensen's inequality yields
  \begin{align*}
    \frac{c n}{b^2} \cdot \E_\mu\left[\ltwo{\wb{Y}_n}^2\right]
    & \le \frac{1}{\alpha} \left(O(1) b \sqrt{\frac{d}{n}}
    + \E_\mu\left[\ltwo{\wb{Y}_n}^2\right]^{1/2} \right)
    + (d+1)\log 2.
  \end{align*}
  By inspection,
  this gives an implied quadratic in $\beta = \E_\mu[\ltwos{\wb{Y}_n}^2]^{1/2}$,
  that is,
  \begin{align*}
    c n b^{-2} \beta^2 
    - \frac{1}{\alpha} \beta 
    - O(1) \frac{b}{\alpha}\sqrt{d/n}
    - (d+1)\log 2
    \le 0;
  \end{align*}
  solving this gives
  \begin{align*}
    \beta \le
    O(1) b^2 \frac{\alpha^{-1} + 
    \sqrt{\alpha^{-2}
    + \alpha^{-1} b^{-1} \sqrt{dn}
    + b^{-2}dn}
    }{n}
    \le 
    O(1) b^2 \left(\frac{1}{\alpha n} 
    + \frac{1}{b} \sqrt{\frac{d}{n}}
    \right),
  \end{align*}
  where we used that
  $2 \alpha^{-1} b^{-1} \sqrt{dn}
  \le \alpha^{-2} + b^{-2} dn$.
  Notice that
  $\E_\mu[\ltwo{\wb{Y}_n}]
  \le \E_\mu[\ltwos{\wb{Y}_n}^2]^{1/2}
  = \beta$ by Jensen's inequality.
\end{proof}

\subsection{Finalizing the proof by conditioning}

Substituting the bounds that Lemma~\ref{lemma:laplace-F-transfer} provides
in inequality~\eqref{eqn:F-with-densities}, making the trivial changes to
deal with only summing over $m \le n$ indices, we see that
if $Y_i$ are independent and satisfy
$\ltwo{Y_i} \le b$, then whenever $F$ has interactions of order at most $d$,
\begin{align*}
  \E\left[F^2 q(\wb{Y}_{n,m} + v) \ltwo{\wb{Y_{n,m}} + w}\right]
  \le \left(\ltwo{w} + O(1) \frac{b^2}{\alpha n}
  + O(1) b \sqrt{\frac{d}{n}}\right)
  \E\left[F^2 q(\wb{Y}_{n,m} + v)\right]
\end{align*}
for any vector $v$.
We can now bound the expectation of the integrands in the revisited area
formula~\eqref{eqn:area-formula-revisited} by performing a conditioning
argument: let $\hat{J} \subset [n]$, $|\hat{J}| \le |J|$, be the sample
indices in $\statrv_1, \ldots, \statrv_n$ associated with $J \in \mc{J}_d$
as in the expansions~\eqref{eqn:expand-in-terms-of-J}.
Then we may define
\begin{align*}
  w = \frac{1}{n} \sum_{i \in \hat{J}}
  \left(A_{\statrv_i}(x) - A(x)\right)
  ~~ \mbox{with} ~~
  \ltwo{w} \le \frac{2 b |\hat{J}|}{n}
  \le \frac{2 b d}{n}
\end{align*}
and
\begin{align*}
  v = v_J(x) = T_{n,\lambda,\tau}(x) - \frac{1}{n} \sum_{i \not\in \hat{J}}
  (A_{\statrv_i}(x) - A(x)),
\end{align*}
where $v$ is $\statrv_{\hat{J}}$-measurable.
Thus, taking an expectation over $\statrv_1^n$ and
conditioning on $\statrv_{\hat{J}}$, Lemma~\ref{lemma:laplace-F-transfer} implies
\begin{align*}
  \lefteqn{\E\left[F_J^2(x \mid P_n) q(T_{n,\lambda,\tau}(x))
      \ltwo{A_n(x) - A(x)}\right]} \\
  & = \E\left[\E\left[F_J^2(x \mid P_n) q(T_{n,\lambda,\tau}(x))
      \ltwo{A_n(x) - A(x)} \mid \statrv_{\hat{J}}\right]\right] \\
  & \le O(1) \left(\frac{b^2}{\alpha n}
  + b \sqrt{\frac{d}{n}}\right)
  \E\left[F_J^2(x \mid P_n) q(T_{n,\lambda,\tau}(x))\right].
\end{align*}

Because each term in the area formula~\eqref{eqn:area-formula-revisited}
is positive, we can interchange integration and expectation, so
\begin{align*}
  \E\left[\ltwo{A_n(T_{n,\lambda,\tau}^{-1}(Z))
      - A(T_{n,\lambda,\tau}^{-1}(Z))}\right]
  & \lesssim \left(\frac{b^2}{\alpha n}
  + b \sqrt{\frac{d}{n}}\right)
  \E\bigg[\int_{\R^d} \sum_{J \in \mc{J}_d}
    F_J^2(x \mid P_n) q(T_{n,\lambda,\tau}(x)) dx \bigg] \\
  & = \left(\frac{b^2}{\alpha n}
  + b \sqrt{\frac{d}{n}}\right)
  \E\left[\int_{\R^d} \det(\deriv T_{n,\lambda,\tau}(x)) q(T_{n,\lambda,\tau}(x))
    dx \right].
\end{align*}
Making the change of variables $z = T_{n,\lambda,\tau}(x)$ shows that
$\int \det(\deriv T_{n,\lambda,\tau}(x)) q(T_{n,\lambda,\tau}(x)) dx
= \int q(z) dz = 1$.
This completes the proof of
Proposition~\ref{proposition:smooth-monotone-laplace}.

\section{Proof of Theorem~\ref{theorem:lower-bound}}
\label{sec:proof-lower-bound}

\newcommand{\maskmat}{M}
\newcommand{\poploss}{f}
\renewcommand{\loss}{F}

We first give a lower bound on distances to stationarity by
function values, which we can exploit to prove the lower bound.
We claim that for any convex function $h$, we have
\begin{align}
  \label{eqn:distances-from-slopes}
  \dist(0, \partial h(x) + \normalcone_X(x)) \ge
  \sup_{y \neq x, y \in X} \frac{\hinge{h(x) - h(y)}}{\ltwo{x - y}}.
\end{align}
To see this, simply observe that for any $g_x \in \partial h(x)$ and
$v \in \normalcone_X(x)$ and $y \in X$,
the first-order convexity inequality
implies $h(x) + \<g_x + v, y - x\> \le h(y)$, so
$h(x) - h(y) \le \ltwo{g_x + v} \ltwo{y - x}$.
Dividing by $\ltwo{y - x}$ gives
inequality~\eqref{eqn:distances-from-slopes}.

As $X$ has interior, by shifting we may without
loss of generality assume that there
is a box $B = [-r, r]^d \subset X$, where $r > 0$.
We adopt an approach that \citet{ShalevShSrSr10} develop in a study of
distribution-free minimization to consider masked $\ell_2$-style losses,
where for $w \in \{-r, r\}^d$ and a masking matrix $\maskmat \in
\diag(\{0,1\}^d)$, using the data $\statval = (\maskmat, \maskmat w)$ we
define the loss
\begin{align*}
  F_\statval(x) = \ltwo{\maskmat (x - w)} + \dist(x, B).
\end{align*}
Evidently, $F_\statval$ is $2$-Lipschitz.
Let $\maskmat$ have i.i.d.\ diagonal with
$\P(\maskmat_{jj} = 1) = p$ for each $j$, where $p \in (0, 1)$ is
to be chosen, whence the population loss
$f_w(x) \defeq \E[\ltwo{\maskmat(x - w)}] + \dist(x, B)$ has
unique minimizer $x\opt = w$.

For this family, we first claim it is (essentially) no loss of
generality to assume that the estimator $\what{x}_n$ lies in $B$.
Indeed, let $\pi_B(x)$ be the (Euclidean) projection of $x$ onto the
set $B$.
Then immediately we see that $\poploss_w(\pi_B(x)) \le
\poploss_w(x)$, while inequality~\eqref{eqn:distances-from-slopes}
implies
\begin{align*}
  \dist(0, \partial \poploss_w(x) + \normalcone_X(x))
  \ge \frac{\poploss_w(x)
    - \poploss_w(\pi_B(x))}{\ltwo{x - \pi_B(x)}}
  \ge \frac{\dist(x, B)}{\ltwo{x - \pi_B(x)}} = 1.
\end{align*}
Any estimator outputting a point $\what{x}_n \not \in B$
evidently has stationary residual at least $1$,
while $\dist(0, \partial \poploss_w(\pi_B(x)) + \normalcone_X(\pi_B(x)))
= \dist(0, \partial \poploss_w(\pi_B(x))) \le 1$ as
$\pi_B(x) \in \interior X$
and the masked loss is $1$-Lipschitz.

For the remainder of the lower bound, we therefore consider points
$x \in B$.
For nonnegative random variables $V$,
by H\"{o}lder's inequality, $\E[V] = \E[V^{1/3} V^{2/3}]
\le \E[\sqrt{V}]^{2/3} \E[V^2]^{1/3}$, or
\begin{align*}
  \E[\sqrt{V}] \ge \frac{\E[V]^{3/2}}{\sqrt{\E[V^2]}}.
\end{align*}
For any $x$, taking $V = \ltwo{\maskmat (x - w)}^2$ with masking
matrix $\maskmat$ with independent diagonal entries satisfying
$\P(\maskmat_{jj} = 1) = p$, we thus observe that
\begin{align*}
  \E[\ltwo{\maskmat(x - w)}]
  \ge \frac{\E[\ltwo{\maskmat(x - w)}^2]^{3/2}}{
    \sqrt{\E[\ltwo{\maskmat(x - w)}^4]}}
  = \frac{p^{3/2} \ltwo{x - w}^3}{
    \sqrt{p(1 - p) \norm{x - w}_4^4 + p^2 \ltwo{x - w}^4}}.
\end{align*}
For $x \in B$, we have
$\norm{x - w}_4^4 \le 4r^2 \ltwo{x - w}^2$, and so
for such $x$ we obtain
\begin{align}
  \nonumber
  \dist(0, \partial \poploss_w(x) + \normalcone_X(x))
  \ge
  \frac{\E[\ltwo{\maskmat(x - w)}]}{\ltwo{x - w}}
  & \ge \frac{p^{3/2} \ltwo{x - w}^2}{
    \sqrt{4p(1 - p) r^2
      \ltwo{x - w}^2 + p^2 \ltwo{x - w}^4}} \\
  & \ge \frac{p \ltwo{x - w}}{
    \sqrt{p \ltwo{x - w}^2 + 4(1 - p) r^2}}.
  \label{eqn:intermediate-lower-bound}
\end{align}

To prove the minimax bound, let $W \sim \uniform\{-r, r\}^d$, and
consider conditionally i.i.d.\ observations $\statrv_i = (\maskmat^{(i)},
\maskmat^{(i)} W)$, $i = 1, \ldots, n$ (conditional on $W$).
The maximum risk is at least the average risk over this family, so we may
use the standard reduction~\cite[Ch.~9]{Duchi26} to assume w.l.o.g.\ that
$\what{x}_n$ is a deterministic function of its observations.
Letting $O$ denote the observed indices, then
conditional on $O$, the coordinates $W_j \simiid \uniform\{-r, r\}$
for $j \not \in O$.
Thus for any estimator $\what{x}_n$, conditional on $O$, we have
\begin{align*}
  \ltwos{\what{x}_n - W}^2 \ge r^2 \floor{\frac{d - |O|}{2}}
\end{align*}
with probability at least $\half$ by the standard Binomial calculation
that if $Y \sim \binomial(n, p)$, then $\P(Y \ge \floor{np}) \ge \half$
and $\P(Y \le \ceil{np}) \ge \half$.
Relatedly,
for any $x$, if $\ltwo{x - w}^2 \ge \frac{1}{3 p} r^2$, then
$p \ltwo{x - w}^2 + 4(1 - p) r^2 \le 13 p \ltwo{x - w}^2$, and so
\begin{align*}
  \dist(0, \partial \poploss_W(\what{x}_n) + \normalcone_X(\what{x}_n))
  & \ge \frac{p \ltwo{x - w}}{\sqrt{13 p \ltwo{x - w}^2}}
  = \frac{1}{\sqrt{13}} \sqrt{p}.
\end{align*}

For $p \in (0, \half]$, we therefore define the event
\begin{align*}
  \mc{E} \defeq \left\{\ltwo{\what{x}_n - W}^2 \ge \frac{1}{3 p} r^2 \right\},
\end{align*}
on which we evidently have
$\dist(0, \partial \poploss_W(\what{x}_n)
+ \normalcone_X(\what{x}_n)) \ge \frac{1}{\sqrt{13}}\sqrt{p}$
and
\begin{align*}
  \E\left[\dist(0, \partial \poploss_W(\what{x}_n)
    + \normalcone_X(\what{x}_n))\right]
  \ge \frac{1}{\sqrt{13}} \sqrt{p} \P(\mc{E}).
\end{align*}
Performing the obvious conditional calculation, we have
\begin{align*}
  \P(\mc{E})
  \ge \P(\mc{E}, d - |O| \ge \frac{1}{p})
  \ge \P\left(\mc{E} \mid d - |O| \ge \frac{1}{p}\right)
  \P\left(|O| \le d - \frac{1}{p}\right)
  \ge \half \P\left(|O| \le d - \frac{1}{p}\right).
\end{align*}
The set of observed indices $O$ satisfies
$|O| \sim \binomial(d, 1 - (1 - p)^n)$,
and so
the median of $|O|$ is at most $d - \frac{1}{p}$
whenever $d$, $p$, and $n$ are such that
\begin{align*}
  d (1 - p)^n \ge \frac{1}{p}.
\end{align*}

We consider two cases, both under the assumption that
$d/n \ge e^2$ and $n \ge 3$.
(Otherwise the minimax bound $1/\sqrt{n}$ is trivial.)
In the first, we assume
$\log(d/n) \le n$.
In this case,
$p = \frac{\log(d/n)}{2 n}$ satisfies
$p \le \half$, and so
\begin{align*}
  1 - p \ge e^{-2p}
  ~~ \mbox{whence} ~~
  (1 - p)^n \ge \exp\left(-\log\frac{d}{n}\right)
  = \frac{n}{d}.
\end{align*}
Then evidently $d(1 - p)^n \ge n \ge \frac{2n}{\log(d/n)} \cdot
\frac{\log \frac{d}{n}}{2} \ge \frac{1}{p}$.
Otherwise, if $d$ is so large that $\log(d/n) > \frac{n}{2}$,
i.e., $d \ge n e^{n/2}$,
we take $p = 1 - e^{-1/2}$, whence
$d (1 - p)^n
= d e^{-n/2} \ge n \ge \frac{1}{p}$.
We see that
\begin{align*}
  \E\left[\dist(0, \partial \poploss_W(\what{x}_n) + \normalcone_X(\what{x}_n)
    )\right]
  \gtrsim \sqrt{p} \P(\mc{E}) \gtrsim \sqrt{\frac{\log(d/n)}{n}} \wedge 1
\end{align*}
as desired.

\subsection*{Statement on AI Use}
We used GPT-6 Astra heavily to develop the proof of
Proposition~\ref{proposition:smooth-monotone-laplace}.
It introduced us to the expansion of the determinant $\det$ in terms of the
Pfaffians in response to queries about expansions involving subsets of
terms in the matrix, and it suggested using the entropy method.
We used no AI to write the actual paper, and we did not use AI to develop
Theorem~\ref{theorem:lower-bound}.

\appendix

\section{Technical proofs}

\subsection{Proof of Lemma~\ref{lemma:interaction-order-moments}}
\label{sec:proof-interaction-order-moments}

Let $F$ have Hoeffding decomposition~\eqref{eqn:hoeffding-decomposition},
\begin{align*}
  F = \sum_{|J| \le k} F_J
  = \sum_{d = 0}^k \sum_{|J| = d} F_J
  = \sum_{d = 0}^k F_d,
\end{align*}
where $F_d = \sum_{|J| = d} F_J$ consists of the sum of the $d$th order
decomposition elements.
Let $\mc{H}_J$ be the Hoeffding spaces of measurable
functions of $Y_J = \{Y_i, i \in J\}$ with
$\E[f \mid Y_K] = 0$ for $|K| < |J|$, as in the definition of the
Hoeffding decomposition,
and let $\mc{H}_k = \oplus_{|J| = k} \mc{H}_J$ be the direct sum
of the Hoeffding spaces of order $k$.
Then $\mc{H}_k \perp \mc{H}_{k'}$ whenever $k \neq k'$.

Because it is useful to consider vector-valued $F$,
for vectors $u \in \R^d$, $v$, we define the tensor product
\begin{align*}
  u \otimes v = (u_1 v, \ldots, u_d v)
  ~~ \mbox{with} ~~
  \ltwo{u \otimes v} = \ltwo{u} \ltwo{v}.
\end{align*}
With this notation,
\begin{align*}
  \ltwo{S_n}^2 \ltwo{F(Y_1^n)}^2
  = \ltwo{S_n \otimes F}^2,
\end{align*}
and so we investigate the random vector $S_n \otimes F$.
For $J \subset [n]$, consider the term
$Y_i \otimes F_J$.
Then we have two cases.
\begin{enumerate}[label=\roman*,leftmargin=*]
\item If $i \not \in J$, then $Y_i \otimes F_J \in \mc{H}_{J \cup \{i\}}$
  because $\E[Y_i] = 0$.
\item If $i \in J$, then
  \begin{align*}
    Y_i \otimes F_J
    = \E[Y_i \otimes F_J \mid Y_{\setminus i}]
    + (Y_i \otimes F_J - \E[Y_i \otimes F_J \mid Y_{\setminus i}])
    = \underbrace{P_i (Y_i \otimes F_J)}_{\in \mc{H}_{J \setminus \{i\}}}
    + \underbrace{(\id - P_i) (Y_i \otimes F_J)}_{\in \mc{H}_J}.
  \end{align*}
\end{enumerate}
Thus we can always write
\begin{align*}
  S_n \otimes F_J
  = \sum_{i \not\in J} Y_i \otimes F_J
  + \sum_{i \in J} (\id - P_i) (Y_i \otimes F_J)
  + \sum_{i \in J} P_i(Y_i \otimes F_J),
\end{align*}
where the three terms increase, leave identical, and decrease the
order of the terms in the Hoeffding decomposition.
For an order $d$ and $F_d = \sum_{|J| = d} F_J$,
we may thus identify
\begin{align*}
  S_n \otimes F_d =
  \underbrace{\sum_{|J| = d} \sum_{i \not \in J} Y_i \otimes F_J}_{\eqdef G_d^+}
  + \underbrace{\sum_{|J| = d}
    \sum_{i \in J} (\id - P_i) (Y_i \otimes F_J)}_{\eqdef G_d^0}
  + \underbrace{\sum_{|J| = d} \sum_{i \in J} P_i(Y_i \otimes F_J)}_{
    \eqdef G_d^-}.
\end{align*}
Thus $G_d^+$, $G_d^0$, and $G_d^-$ increase, leave the same, and
decrease the order of the Hoeffding decomposition, so that
$G_d^+ \in \mc{H}_{d+1}$, $G_d^0 \in \mc{H}_d$, and
$G_d^- \in \mc{H}_{d-1}$.

We control each of the three in turn.
For the first, rearranging the sum over $K = J \cup \{i\}$ for
$i \not \in J$, we evidently have
\begin{align*}
  G_d^+ = \sum_{|K| = d + 1} \sum_{i \in K} Y_i \otimes F_{K \setminus \{i\}}
\end{align*}
and $Y_i \otimes F_{K \setminus \{i\}} \in \mc{H}_K$, and as the spaces
$\mc{H}_K$ are orthogonal,
\begin{align*}
  \E\left[\ltwo{G_d^+}^2\right]
  = \sum_{|K| = d + 1} \E\left[\ltwobigg{\sum_{i \in K} Y_i \otimes
      F_{K \setminus \{i\}}}^2\right]
  & \le \sum_{|K| = d + 1}
  (d + 1) \sum_{i \in K} \E\left[\ltwo{Y_i \otimes F_{K \setminus i}}^2\right] \\
  & \le (d + 1) \sum_{|K| = d + 1}
  \sum_{i \in K} \E\left[\ltwo{F_{K \setminus i}}^2\right]
\end{align*}
by Cauchy-Schwarz and that $\ltwo{Y_i} \le 1$.
A counting argument shows that for each set $J$ of size $d$,
we have $J = K \setminus \{i\}$ for a set $K$ once
for each $i \not \in J$, of which there are $n - d$, so
$\sum_{|K| = d + 1} \sum_{i \in K} \E[\ltwos{F_{K \setminus i}}^2]
= (n - d) \sum_{|J| = d} \E[\ltwos{F_J}^2]$.
We conclude that
\begin{align}
  \label{eqn:control-Gd-plus}
  \E\left[\ltwo{G_d^+}^2\right]
  & \le (d + 1) (n - d)
  \sum_{|J| = d} \E\left[\ltwo{F_J}^2\right]
  = (d + 1) (n - d) \E\left[\ltwo{F_d}^2\right]
\end{align}
by orthogonality of the $F_J$.

A similar calculation applies for $G_d^0$ and $G_d^-$.
For the latter, we again rearrange the sum to write
\begin{align*}
  G_d^- = \sum_{|K| = d - 1} \sum_{i \not \in K}
  P_i (Y_i \otimes F_{K \cup \{i\}}),
\end{align*}
where the elements $P_i (Y_i \otimes F_{K \cup \{i\}}) \in \mc{H}_K$, and so
\begin{align*}
  \lefteqn{\E\left[\ltwo{G_d^-}^2\right]
    = \sum_{|K| = d - 1}
    \E\left[\ltwobigg{\sum_{i \not \in K}
        P_i (Y_i \otimes F_{K \cup \{i\}})}^2 \right]} \\
  & \quad \le (n - d + 1) \sum_{|K| = d - 1}
  \sum_{i \not \in K}
  \E\left[\ltwo{P_i (Y_i \otimes F_{K \cup \{i\}})}^2\right]
  \le (n - d + 1) \sum_{|K| = d - 1}
  \sum_{i \not \in K} \E\left[\ltwo{F_{K \cup \{i\}}}^2\right],
\end{align*}
where the first inequality uses Cauchy-Schwarz and the second that $P_i$
is a (Hilbert) projection in $L^2$.
Again reindexing, the final sum is
$d \sum_{|J| = d} \E[\ltwos{F_J}^2]$, so that
\begin{align}
  \label{eqn:control-Gd-minus}
  \E\left[\ltwo{G_d^-}^2\right]
  & \le d (n - d + 1) \sum_{|J| = d} \E\left[\ltwo{F_J}^2\right]
  = d (n - d + 1) \E\left[\ltwo{F_d}^2\right].
\end{align}
Finally,
to control $G_d^0$, observe
that
\begin{align*}
  \E\left[\ltwo{G_d^0}^2\right] & = \sum_{|J| = d}
  \E\left[\ltwobigg{\sum_{i \in J} (\id - P_i) (Y_i \otimes F_J)}^2
    \right]
  \le d \sum_{|J| = d}
  \sum_{i \in J} \E\left[\ltwo{(\id - P_i) (Y_i \otimes F_J)}^2\right] \\
  & \le d \sum_{|J| = d} \sum_{i \in J}
  \E\left[\ltwo{F_J}^2\right]
\end{align*}
because projections are contractions and $\ltwo{Y_i} \le 1$.
So $\E[\ltwos{G_d^0}^2] \le d^2 \E[\ltwos{F_d}^2]$.

Combining inequalities~\eqref{eqn:control-Gd-plus},
\eqref{eqn:control-Gd-minus}, and the final bound $\E[\ltwos{G_d^0}^2] \le
d^2 \E[\ltwo{F_d}^2]$, then using the orthogonality of the spaces
$\mc{H}_d$ implies
\begin{align*}
  \E\left[\ltwo{S_n F_d}^2\right]
  & \le \left((d + 1) (n - d) + d(n - d + 1) + d^2\right)
  \E\left[\ltwo{F_d}^2\right] \\
  & = \left((2d + 1) n - d^2\right) \E\left[\ltwo{F_d}^2\right]
  \le (2d + 1) n \E\left[\ltwo{F_d}^2\right].
\end{align*}
Finally, we use that $\mc{H}_d \perp \mc{H}_{d'}$ for $d \neq d'$ to
obtain
\begin{align*}
  \E\left[\ltwo{S_n F}^2\right]
  & = \E\left[\ltwobigg{\sum_{d = 0}^k (G_d^+ + G_d^0 + G_d^-)}^2
    \right]
  \le O(1) \cdot (k + 1) n \E\left[\ltwo{F}^2\right],
\end{align*}
because $\E[\<G_d^{\pm}, G_{d'}^{\pm}\>] = 0$ whenever $|d - d'| \ge 2$,
as desired.

\subsection{Proof of Lemma~\ref{lemma:det-as-pfaffian}}
\label{sec:proof-det-as-pfaffian}

First, we assume that $B$ is invertible.
Then applying the matrix determinant lemma,
$\det(B + VV^\top) = \det(B) \det(I_n + V^\top B^{-1} V)$, and
using the principal-minor expansion of the characteristic polynomial
of a matrix~\cite[Section~1.2]{HornJo12} we have
\begin{align*}
  \det(I_n + V^\top B^{-1} V)
  = \sum_{J \subset [n]}
  \det(V_J^\top B^{-1} V_J).
\end{align*}
Because $V_J$ is rank deficient whenever $|J| > d$,
we can rewrite this as
\begin{align*}
  \det(B + VV^\top)
  = \sum_{|J| \le d} \det(B) \det(V_J^\top B^{-1} V_J)
  = \sum_{|J| \le d}
  \det\left(\begin{matrix} B & V_J \\ -V_J^\top & 0 \end{matrix}
  \right),
\end{align*}
where the final equality follows via the standard
block determinant identity.
Noting that the left and right side are both polynomials in $B$,
independent of $B^{-1}$, we can employ a trivial approximation argument
to argue that the identity extends to arbitrary matrices $B$.

Lastly, let us use the assumption that $B$ is skew-symmetric.
In this case, the matrices
\begin{align*}
  C_J \defeq \left[\begin{matrix} B & V_J \\ -V_J^\top & 0 \end{matrix}
    \right]
\end{align*}
are also skew-symmetric,
so that
\begin{align*}
  \det(B + VV^\top) = \sum_{|J| \le d} \det(C_J)
  = \sum_{|J| \le d} \pf(C_J)^2.
\end{align*}
To obtain the degree argument, first recognize that
if $d + |J|$ is odd, then
$C_J$ is odd-dimensional and so has Pfaffian zero:
it necessarily has a zero eigenvalue, because its non-zero eigenvalues are
pure imaginary.
So we can restrict the sum to $J \in \mc{J}_d$.
Finally, by definition~\eqref{eqn:pfaffian-definition}, for fixed $V_J$,
the block zero structure in the matrix $C_J$ indicates that in each
permutation $\sigma \in \mathfrak{S}_{d + |J|}$, each index in $J$ (i.e.,
in $\{d+1, \ldots, d + |J|\}$) must pair with one of the indices
$\{1, \ldots, d\}$, while
the remaining indices in $\{1, \ldots, d\}$ corresponding to $B$ must pair
with one another; there are $d - |J|$ such indices, and the product
structure~\eqref{eqn:pfaffian-definition} shows that $(d - |J|)/2$ terms
of $B$ appear.

\bibliography{bib}
\bibliographystyle{abbrvnat}

\end{document}